\documentclass[11pt]{amsart}

\usepackage{amscd}
\usepackage{amsmath, amssymb}
\usepackage{amsfonts}
\newcommand{\de}{\partial}

\newcommand{\ti}[1]{\tilde{#1}}

\newcommand{\ve}{\varepsilon}

\renewcommand{\leq}{\leqslant}
\renewcommand{\geq}{\geqslant}

\newcommand{\be}{\begin{equation}}
\newcommand{\ee}{\end{equation}}

\begin{document}
\newtheorem{claim}{Claim}
\newtheorem{theorem}{Theorem}[section]
\newtheorem{lemma}[theorem]{Lemma}
\newtheorem{corollary}[theorem]{Corollary}
\newtheorem{proposition}[theorem]{Proposition}
\newtheorem{question}{question}[section]
\newtheorem{definition}[theorem]{Definition}
\newtheorem{remark}[theorem]{Remark}

\numberwithin{equation}{section}

\title[Curvature estimate]
{A simple proof of curvature estimate for convex solution of $k$-Hessian equation}

\author[J. Chu]{Jianchun Chu}
\address{Department of Mathematics, Northwestern University, 2033 Sheridan Road, Evanston, IL 60208}
\email{jianchun@math.northwestern.edu}

\begin{abstract}
Guan-Ren-Wang \cite{GRW15} established the curvature estimate of convex hypersurface satisfying the Weingarten curvature equation $\sigma_{k}(\kappa(X)) = f(X,\nu(X))$. In this note, we give a simple proof of this result.
\end{abstract}

\maketitle

\section{Introduction}
Let $M \subset \mathbb{R}^{n+1}$ be a closed hypersurface. We consider the following curvature equation in a general form:
\begin{equation}\label{Curvature equation}
\sigma_k (\kappa(X)) = f (X, \nu (X)) \quad \text{for $X\in M$},
\end{equation}
where $\kappa (X)$ and $\nu (X)$ are principal curvatures and unit outer normal vector at $X \in M$, and $\sigma_{k}$ denotes the $k$-th elementary symmetric function
\[
\sigma_{k} (\kappa) = \sum_{i_{1}<i_{2}<\cdots<i_{k}}
\kappa_{i_{1}}\kappa_{i_{2}}\cdots\kappa_{i_{k}}.
\]
For $1\leq k\leq n$, $\sigma_{k}(\kappa)$ are the Weingarten curvatures of $M$. In particular, $\sigma_{1}(\kappa)$, $\sigma_{2}(\kappa)$ and $\sigma_{n}(\kappa)$ are the mean curvature, scalar curvature and Gauss curvature, respectively.

The curvature equation (\ref{Curvature equation}) plays a significant role in geometry. Many important geometric problem can be transformed into (\ref{Curvature equation}) with a special form of $f$, including the Minkowski problem (\cite{Nirenberg53,Pogorelov53,Pogorelov78,CY76}), the problem of prescribing general Weingarten curvature on outer normals by Alexandrov (\cite{Aleksandrov56,GG02}), the problem of prescribing curvature measures in convex geometry \cite{Alexandroff42,Pogorelov53,GLM09,GLL12}) and the prescribing curvature problem considered in \cite{BK74,TW83,CNS86}.

The curvature equation (\ref{Curvature equation}) has been studied extensively. When $k=1$, equation (\ref{Curvature equation}) is quasi-linear, so the curvature estimate follows from the classical theory of quasi-linear PDEs. When $k=n$, equation (\ref{Curvature equation}) is of Monge-Amp\`ere type. The desired estimate was established by Caffarelli-Nirenberg-Spruck \cite{CNS84}.

When $1<k<n$, if $f$ is independent of $\nu$, Caffarelli-Nirenberg-Spruck \cite{CNS86} established the curvature estimate; if $f$ depends only on $\nu$, the curvature estimate was proved by Guan-Guan \cite{GG02}. In \cite{Ivochkina89,Ivochkina90}, Ivochkina studied the Dirichlet problem of equation (\ref{Curvature equation}) on domains in $\mathbb{R}^{n}$ and obtained the curvature estimate under some additional assumptions on the dependence of $f$ on $\nu$. For the prescribing curvature measures problem, Guan-Lin-Ma \cite{GLM09} and Guan-Li-Li \cite{GLL12} proved the curvature estimate for $f(X,\nu)=\langle X,\nu\rangle\ti{f}(X)$.

For general right-hand side $f(X,\nu)$, establishing the curvature estimate for equation (\ref{Curvature equation}) is very important and interesting in both geometry and PDEs. In \cite{GRW15}, Guan-Ren-Wang solved the case $k=2$ (in \cite{SX17}, Spruck-Xiao gave a simplified proof). In \cite{RW19,RW20}, Ren-Wang solved the cases $k=n-1$ and $k=n-2$. The other cases $2<k<n-2$ are still open. For general $k$, Guan-Ren-Wang \cite{GRW15} established the following curvature estimate for convex hypersurface:

\begin{theorem}\label{Curvature estimate}[Guan-Ren-Wang, \cite[Theorem 1.1]{GRW15}]
Let $M$ be a closed convex hypersurface satisfying curvature equation (\ref{Curvature equation})
for some positive function $f\in C^{2}(\Gamma)$, where $\Gamma$ is an open neighborhood of the unit normal bundle
of $M$ in $\mathbb{R}^{n+1}\times\mathbb{S}^{n}$. There exists a constant $C$ depending only
$n$, $k$, $\|M\|_{C^{1}}$, $\inf f$ and $\|f\|_{C^{2}}$ such that
\[
\max_{X\in M,~i=1,2,\cdots,n}\kappa_{i}(X) \leq C.
\]
\end{theorem}

In this note, we give a simple proof of Theorem \ref{Curvature estimate}. Compared to \cite{GRW15}, we take a different approach. To establish the curvature estimate, the main difficulty is how to deal with the third order terms. We apply the maximum principle to a quantity involving the largest principal curvature $\kappa_{1}$, instead of the symmetric function of $\kappa$. This gives us more ``good" third order terms, and so simplifies the argument.

\section{Preliminaries}
For any point $X_{0}\in M$, let $\{e_{i}\}_{i=1}^{n}$ be a local orthonormal frame near $X_{0}$ such that
\[
h_{ij} = \delta_{ij}\kappa_{i}, \quad
\kappa_{1} \geq \kappa_{2} \geq \cdots \geq \kappa_{n} \ \ \text{at $X_{0}$}.
\]
We use the following notations:
\[
\sigma_{k}^{ij} = \frac{\de\sigma_{k}}{\de h_{ij}}, \quad
\sigma_{k}^{ij,pq} = \frac{\de^{2}\sigma_{k}}{\de h_{ij}\de h_{pq}}.
\]
Then at $X_{0}$, we have (see e.g. \cite{Gerhardt96,Spruck05})
\[
\sigma_{k}^{ij} = \sigma_{k-1}(\kappa|i)\delta_{ij}
\]
and
\[
\sigma_{k}^{ij,pq} =
\begin{cases}
\sigma_{k-2}(\kappa|ip) &  \mbox{if $i=j$, $p=q$, $i\neq p$};  \\
-\sigma_{k-2}(\kappa|ip) &  \mbox{if $i=q$, $p=j$, $i\neq p$};  \\
0            &  \mbox{otherwise},
\end{cases}
\]
where $\sigma_{s}(\kappa|i_{1}\cdots i_{r})$ denotes $s$-th elementary symmetric function with $\kappa_{i_{1}}=\kappa_{i_{2}}=\cdots=\kappa_{i_{r}}=0$.

Here we list some well-known formulas:
\[
\begin{split}
\text{Guass formula}: \quad & X_{ij} = -h_{ij}\nu, \\
\text{Weingarten equation}: \quad & \nu_{i} = h_{ij}e_{j}, \\
\text{Codazzi formula}: \quad & h_{ijp} = h_{ipj}, \\
\text{Guass equation}: \quad & R_{ijpq} = h_{ip}h_{jq}-h_{iq}h_{jp},
\end{split}
\]
where $R_{ijpq}$ is the curvature tensor of $M$. We also have
\begin{equation}\label{Commutation formula}
h_{pqij} = h_{ijpq}+(h_{mq}h_{pj}-h_{mj}h_{pq})h_{mi}+(h_{mq}h_{ij}-h_{mj}h_{iq})h_{mp}.
\end{equation}

\section{Simple proof of Theorem \ref{Curvature estimate}}
In this section, we give a simple proof of Theorem \ref{Curvature estimate}.
\begin{proof}[Simple proof of Theorem \ref{Curvature estimate}]
Since $M$ is convex, after shifting the origin of $\mathbb{R}^{n+1}$, we assume that $M$ is star-shaped with respect to the new origin. Thus the support function $u(X)=\langle X,\nu(X)\rangle$ is always positive. By assumptions, there exists a uniform constant $C>0$ such that
\[
\frac{1}{C} \leq u \leq C \quad \text{for $X\in M$}.
\]
Let $\kappa_{1}$ be the largest principal curvature. Since $M$ is convex, to prove Theorem \ref{Curvature estimate}, it suffices to prove $\kappa_{1}$ is uniformly bounded from above. Without loss of generality, we assume that the set $\Omega=\{\kappa_{1}>0\}$ is not empty. On $\Omega$, we consider the following quantity
\[
Q = \log\kappa_{1}-Au,
\]
where $A>1$ is a constant to be determined later. Note that $Q$ is continuous on $\Omega$, and goes to $-\infty$ on $\de\Omega$. Hence $Q$ achieves a maximum at a point $X_{0}$ with $\kappa_{1}(X_{0})>0$. However, the function $Q$ may be not smooth at $X_{0}$ when the eigenspace of $\kappa_{1}$ has dimension strictly larger than $1$, i.e., $\kappa_{1}=\kappa_{2}$ at $X_{0}$. To deal with this case, we apply the standard perturbation argument. Let $g$ be the first fundamental form of $M$ and $D$ be the corresponding Levi-Civita connection. We choose a local orthonormal frame $\{e_{i}\}_{i=1}^{n}$ near $X_{0}$ such that
\[
D_{e_{i}}e_{j} = 0, \quad
h_{ij} = \delta_{ij}\kappa_{i}, \quad
\kappa_{1} \geq \kappa_{2} \geq \cdots \geq \kappa_{n} \ \ \text{at $X_{0}$}.
\]
We now apply a perturbation argument. Near $X_{0}$, we define a new tensor $B$ by
\[
B(V_{1},V_{2}) = g(V_{1},V_{2})-g(V_{1},e_{1})g(V_{2},e_{1}),
\]
for tangent vectors $V_{1}$ and $V_{2}$. Let $B_{ij}=B(e_{i},e_{j})$. It is clear that
\[
\quad B_{ij} = \delta_{ij}B_{ii}, \quad B_{11} = 0, \quad B_{ii} = 1 \ \ \text{for $i>1$}.
\]
We define the matrix by $\ti{h}_{ij}=h_{ij}-B_{ij}$, and denote its eigenvalues by $\ti{\kappa}_{1}\geq\ti{\kappa}_{2}\geq\cdots\geq\ti{\kappa}_{n}$. It then follows that $\kappa_{1} \geq \ti{\kappa}_{1}$ near $X_{0}$ and
\[
\ti{\kappa}_{i} =
\begin{cases}
\kappa_{1}   &  \mbox{if $i=1$}, \\
\kappa_{i}-1 &  \mbox{if $i>1$},
\end{cases}
\quad \text{at $X_{0}$}.
\]
Thus $\ti{\kappa}_{1} > \ti{\kappa}_{2}$ at $X_{0}$, which implies that $\ti{\kappa}_{1}$ is smooth at $X_{0}$. We consider the perturbed quantity $\ti{Q}$ defined by
\[
\ti{Q} = \log\ti{\kappa}_{1}-Au,
\]
which still achieves a local maximum at $X_{0}$. From now on, all the calculations will be carried out at $X_{0}$. For any $1\leq i\leq n$, since $\ti{\kappa}_{1}=\kappa_{1}$ at $X_{0}$, we have
\begin{equation}\label{ti Q i}
0 = \ti{Q}_{i}
= \frac{\ti{\kappa}_{1,i}}{\ti{\kappa}_{1}}-Au_{i}
= \frac{\ti{\kappa}_{1,i}}{\kappa_{1}}-Au_{i}
\end{equation}
and
\begin{equation}\label{sigma k ii ti Q ii}
0 \geq \sigma_{k}^{ii}\ti{Q}_{ii}
= \sigma_{k}^{ii}(\log\ti{\kappa}_{1})_{ii}-A\sigma_{k}^{ii}u_{ii}.
\end{equation}

In the following lemma, we estimate each term in (\ref{sigma k ii ti Q ii}) and obtain an inequality.

\begin{lemma}\label{Calculation}
At $X_{0}$, we have
\[
\begin{split}
0 \geq {} & 2\sum_{p>1}\frac{\sigma_{k}^{11,pp}h_{11p}^{2}}{\kappa_{1}}
+2\sum_{p>1}\frac{\sigma_{k}^{11}h_{11p}^{2}}{\kappa_{1}(\kappa_{1}-\ti{\kappa}_{p})}
-\frac{\sigma_{k}^{pp,qq}h_{pp1}h_{qq1}}{\kappa_{1}} \\
& +2\sum_{p>1}\frac{\sigma_{k}^{pp}h_{pp1}^{2}}{\kappa_{1}(\kappa_{1}-\ti{\kappa}_{p})}
-\frac{\sigma_{k}^{pp}h_{11p}^{2}}{\kappa_{1}^{2}}+\left(\frac{A}{C}-C\right)\sigma_{k}^{ii}h_{ii}^{2}-CA.
\end{split}
\]
\end{lemma}

\begin{proof}
First, let us recall the first and second derivatives of $\ti{\kappa}_{1}$ at $X_{0}$ (see e.g. \cite{Spruck05}):
\[
\begin{split}
\ti{\kappa}_{1}^{pq}
:= {} & \frac{\de\ti{\kappa}_{1}}{\de\ti{h}_{pq}} = \delta_{1p}\delta_{1q}, \\
\ti{\kappa}_{1}^{pq,rs}
:= {} & \frac{\de^{2}\ti{\kappa}_{1}}{\de\ti{h}_{pq}\de\ti{h}_{rs}}
= (1-\delta_{1p})\frac{\delta_{1q}\delta_{1r}\delta_{ps}}{\ti{\kappa}_{1}-\ti{\kappa}_{p}}
+ (1-\delta_{1r})\frac{\delta_{1s}\delta_{1p}\delta_{qr}}{\ti{\kappa}_{1}-\ti{\kappa}_{r}}.
\end{split}
\]
We compute
\[
\begin{split}
\ti{\kappa}_{1,i} = {} & \ti{\kappa}_{1}^{pq}\ti{h}_{pqi} = \ti{h}_{11i}, \\
\ti{\kappa}_{1,ii} = {} & \ti{\kappa}_{1}^{pq}\ti{h}_{pqii}+\ti{\kappa}_{1}^{pq,rs}\ti{h}_{pqi}\ti{h}_{rsi}
= \ti{h}_{11ii}+2\sum_{p>1}\frac{\ti{h}_{1pi}^{2}}{\kappa_{1}-\ti{\kappa}_{p}},
\end{split}
\]
where we used $\ti{\kappa}_{1}=\kappa_{1}$ at $X_{0}$. Using the definition of tensor $B$ and $(D_{e_{i}}e_{j})(X_{0})=0$, we see that
\[
B_{ij,p} = 0, \quad
B_{11,ii} = 0 \ \ \text{at $X_{0}$}.
\]
Combining this with $\ti{h}_{ij}=h_{ij}-B_{ij}$, we obtain
\[
\ti{h}_{ijp} = h_{ijp}, \quad \ti{h}_{11ii} = h_{11ii} \ \ \text{at $X_{0}$}.
\]
It then follows that
\begin{equation}\label{Calculation eqn 1}
\ti{\kappa}_{1,i} = h_{11i}, \quad
\ti{\kappa}_{1,ii} = h_{11ii}+2\sum_{p>1}\frac{h_{1pi}^{2}}{\kappa_{1}-\ti{\kappa}_{p}}.
\end{equation}

For the term $\sigma_{k}^{ii}(\log\ti{\kappa}_{1})_{ii}$ in (\ref{sigma k ii ti Q ii}), using (\ref{Calculation eqn 1}) and $\ti{\kappa}_{1}=\kappa_{1}$ at $X_{0}$, we compute
\begin{equation}\label{Calculation eqn 2}
\begin{split}
\sigma_{k}^{ii}(\log\ti{\kappa}_{1})_{ii}
= {} & \frac{\sigma_{k}^{ii}\ti{\kappa}_{1,ii}}{\ti{\kappa}_{1}}-\frac{\sigma_{k}^{ii}\ti{\kappa}_{1,i}^{2}}{\ti{\kappa}_{1}^{2}} \\
= {} & \frac{\sigma_{k}^{ii}h_{11ii}}{\kappa_{1}}+2\sum_{p>1}\frac{\sigma_{k}^{ii}h_{1pi}^{2}}{\kappa_{1}(\kappa_{1}-\ti{\kappa}_{p})}
-\frac{\sigma_{k}^{ii}h_{11i}^{2}}{\kappa_{1}^{2}}.
\end{split}
\end{equation}
By (\ref{Commutation formula}), we have
\begin{equation}\label{Calculation eqn 10}
\begin{split}
\sigma_{k}^{ii}h_{11ii} = {} & \sigma_{k}^{ii}h_{ii11}+\sigma_{k}^{ii}(h_{i1}^{2}-h_{ii}h_{11})h_{ii}+\sigma_{k}^{ii}(h_{11}h_{ii}-h_{i1}^{2})h_{11} \\
= {} & \sigma_{k}^{ii}h_{ii11}-\sigma_{k}^{ii}h_{ii}^{2}h_{11}+\sigma_{k}^{ii}h_{ii}h_{11}^{2} \\
= {} & \sigma_{k}^{ii}h_{ii11}-\sigma_{k}^{ii}h_{ii}^{2}h_{11}+kfh_{11}^{2}.
\end{split}
\end{equation}
where we used
\begin{equation}\label{Calculation eqn 7}
\sum_{i}\sigma_{k}^{ii}h_{ii}
= \sum_{i}\kappa_{i}\sigma_{k-1}(\kappa|i)
= k\sigma_{k}(\kappa) = kf.
\end{equation}
On the other hand, differentiating (\ref{Curvature equation}) twice, we obtain
\[
\sigma_{k}^{ii}h_{ii11} \geq -\sigma_{k}^{ij,pq}h_{ij1}h_{pq1}+\sum_{p}h_{p11}(d_{\nu}f)(e_{p})-Ch_{11}^{2}-C.
\]
Combining this with (\ref{Calculation eqn 10}),
\[
\sigma_{k}^{ii}h_{11ii} \geq -\sigma_{k}^{ij,pq}h_{ij1}h_{pq1}+\sum_{p}h_{p11}(d_{\nu}f)(e_{p})-\sigma_{k}^{ii}h_{ii}^{2}h_{11}-Ch_{11}^{2}-C.
\]
Substituting this into (\ref{Calculation eqn 2}),
\begin{equation}\label{Calculation eqn 4}
\begin{split}
\sigma_{k}^{ii}(\log\ti{\kappa}_{1})_{ii}
\geq {} & -\frac{\sigma_{k}^{ij,pq}h_{ij1}h_{pq1}}{\kappa_{1}}+2\sum_{p>1}\frac{\sigma_{k}^{ii}h_{1pi}^{2}}{\kappa_{1}(\kappa_{1}-\ti{\kappa}_{p})}
-\frac{\sigma_{k}^{pp}h_{11p}^{2}}{\kappa_{1}^{2}} \\
& +\frac{1}{\kappa_{1}}\sum_{p}h_{p11}(d_{\nu}f)(e_{p})-\sigma_{k}^{ii}h_{ii}^{2}-Ch_{11}-C,
\end{split}
\end{equation}
assuming without loss of generality that $\kappa_{1}\geq 1$.

By Guass formula, Weingarten equation and Codazzi formula, we see that
\[
u_{ii} = \sum_{p}h_{iip}\langle e_{p},X\rangle-uh_{ii}^{2}+h_{ii}.
\]
For the term $-A\sigma_{k}^{ii}u_{ii}$ in (\ref{sigma k ii ti Q ii}), we compute
\begin{equation}\label{Calculation eqn 5}
\begin{split}
-A\sigma_{k}^{ii}u_{ii}
= {} & -A\sum_{p}\sigma_{k}^{ii}h_{iip}\langle e_{p},X\rangle+Au\sigma_{k}^{ii}h_{ii}^{2}-A\sigma_{k}^{ii}h_{ii} \\
\geq {} & -A\sum_{p}\sigma_{k}^{ii}h_{iip}\langle e_{p},X\rangle+\frac{A\sigma_{k}^{ii}h_{ii}^{2}}{C}-Akf,
\end{split}
\end{equation}
where we used $u\geq\frac{1}{C}$ and (\ref{Calculation eqn 7}). Differentiating (\ref{Curvature equation}), we obtain
\[
\sigma_{k}^{ii}h_{iip} = h_{pp}(d_{\nu}f)(e_{p})+(d_{X}f)(e_{p}).
\]
Substituting this into (\ref{Calculation eqn 5}), we have
\begin{equation}\label{Calculation eqn 6}
-A\sigma_{k}^{ii}u_{ii}
\geq -A\sum_{p}h_{pp}(d_{\nu}f)(e_{p})\langle e_{p},X\rangle+\frac{A\sigma_{k}^{ii}h_{ii}^{2}}{C}-CA.
\end{equation}

Combining (\ref{sigma k ii ti Q ii}), (\ref{Calculation eqn 4}) and (\ref{Calculation eqn 6}), we obtain
\begin{equation}\label{Calculation eqn 3}
\begin{split}
0 \geq {} & F^{ii}\ti{Q}_{ii} \\
\geq {} & -\frac{\sigma_{k}^{ij,pq}h_{ij1}h_{pq1}}{\kappa_{1}}+2\sum_{p>1}\frac{\sigma_{k}^{ii}h_{1pi}^{2}}{\kappa_{1}(\kappa_{1}-\ti{\kappa}_{p})}
-\frac{\sigma_{k}^{pp}h_{11p}^{2}}{\kappa_{1}^{2}} \\
& +\frac{1}{\kappa_{1}}\sum_{p}h_{p11}(d_{\nu}f)(e_{p})-A\sum_{p}h_{pp}(d_{\nu}f)(e_{p})\langle e_{p},X\rangle \\
& +\left(\frac{A}{C}-1\right)\sigma_{k}^{ii}h_{ii}^{2}-Ch_{11}-CA.
\end{split}
\end{equation}
Using (\ref{ti Q i}), (\ref{Calculation eqn 1}) and $u_{p}=h_{pp}\langle e_{p},X\rangle$, for $1\leq p\leq n$, we have
\begin{equation}\label{Calculation eqn 11}
\frac{h_{11p}}{\kappa_{1}}-Ah_{pp}\langle e_{p},X\rangle = 0.
\end{equation}
Combining this with Codazzi formula, it is clear that
\begin{equation}\label{Calculation eqn 8}
\frac{1}{\kappa_{1}}\sum_{p}h_{p11}(d_{\nu}f)(e_{p})-A\sum_{p}h_{pp}(d_{\nu}f)(e_{p})\langle e_{p},X\rangle = 0.
\end{equation}
By \cite[Lemma 3.1]{CW01}, we have
\begin{equation}\label{Calculation eqn 9}
\kappa_{1} = h_{11} \leq C\sigma_{k}^{11}h_{11}^{2}.
\end{equation}
Substituting (\ref{Calculation eqn 8}) and (\ref{Calculation eqn 9}) into (\ref{Calculation eqn 3}),
\[
0 \geq  -\frac{\sigma_{k}^{ij,pq}h_{ij1}h_{pq1}}{\kappa_{1}}+2\sum_{p>1}\frac{\sigma_{k}^{ii}h_{1pi}^{2}}{\kappa_{1}(\kappa_{1}-\ti{\kappa}_{p})}
-\frac{\sigma_{k}^{pp}h_{11p}^{2}}{\kappa_{1}^{2}}+\left(\frac{A}{C}-C\right)\sigma_{k}^{ii}h_{ii}^{2}-CA.
\]
Combining this with
\[
\begin{split}
& -\frac{\sigma_{k}^{ij,pq}h_{ij1}h_{pq1}}{\kappa_{1}}+2\sum_{p>1}\frac{\sigma_{k}^{ii}h_{1pi}^{2}}{\kappa_{1}(\kappa_{1}-\ti{\kappa}_{p})} \\
\geq {} & -\frac{\sigma_{k}^{pp,qq}h_{pp1}h_{qq1}}{\kappa_{1}}
+2\sum_{p>1}\frac{\sigma_{k}^{11,pp}h_{11p}^{2}}{\kappa_{1}}
+2\sum_{p>1}\frac{\sigma_{k}^{11}h_{1p1}^{2}}{\kappa_{1}(\kappa_{1}-\ti{\kappa}_{p})}
+2\sum_{p>1}\frac{\sigma_{k}^{pp}h_{1pp}^{2}}{\kappa_{1}(\kappa_{1}-\ti{\kappa}_{p})}
\end{split}
\]
and Codazzi formula, we obtain Lemma \ref{Calculation}.
\end{proof}

\begin{lemma}\label{Lemma 1}
At $X_{0}$, we have
\[
\sum_{p>1}\frac{\sigma_{k}^{pp}h_{11p}^{2}}{\kappa_{1}^{2}}
\leq 2\sum_{p>1}\frac{\sigma_{k}^{11,pp}h_{11p}^{2}}{\kappa_{1}}
+2\sum_{p>1}\frac{\sigma_{k}^{11}h_{11p}^{2}}{\kappa_{1}(\kappa_{1}-\ti{\kappa}_{p})},
\]
assuming without loss of generality that $\kappa_{1}\geq 1$.
\end{lemma}

\begin{proof}
We define
\[
I = \{ p \in \{2,3,\cdots,n\}~|~\kappa_{i}=\kappa_{1}\}.
\]

For $p\in I$, we have $\sigma_{k}^{pp}=\sigma_{k}^{11}$ and $\kappa_{1}-\ti{\kappa}_{p}=1$. Thus,
\begin{equation}\label{Lemma 1 eqn 1}
\sum_{p\in I}\frac{\sigma_{k}^{pp}h_{11p}^{2}}{\kappa_{1}^{2}}
= \frac{1}{\kappa_{1}}\sum_{p\in I}\frac{\sigma_{k}^{11}h_{11p}^{2}}{\kappa_{1}}
\leq \sum_{p\in I}\frac{\sigma_{k}^{11}h_{11p}^{2}}{\kappa_{1}(\kappa_{1}-\ti{\kappa}_{p})}.
\end{equation}

For $p\notin I$, since $\ti{\kappa}_{p}=\kappa_{p}-1$ and $\kappa_{p}>0$, then
\[
\kappa_{1}-\ti{\kappa}_{p} = \kappa_{1}-\kappa_{p}+1 \leq \kappa_{1}+1 \leq 2\kappa_{1},
\]
which implies
\begin{equation}\label{Lemma 1 eqn 2}
\sum_{p\notin I}\frac{\sigma_{k}^{11}h_{11p}^{2}}{\kappa_{1}^{2}}
\leq 2\sum_{p\notin I}\frac{\sigma_{k}^{11}h_{11p}^{2}}{\kappa_{1}(\kappa_{1}-\ti{\kappa}_{p})}.
\end{equation}
On the other hand, using $0<\kappa_{p}<\kappa_{1}$, we have
\[
\frac{\sigma_{k}^{pp}-\sigma_{k}^{11}}{\kappa_{1}^{2}}
\leq \frac{\sigma_{k}^{pp}-\sigma_{k}^{11}}{\kappa_{1}(\kappa_{1}-\kappa_{p})}
= \frac{\sigma_{k}^{11,pp}}{\kappa_{1}}.
\]
It then follows that
\begin{equation}\label{Lemma 1 eqn 3}
\sum_{p\notin I}\frac{(\sigma_{k}^{pp}-\sigma_{k}^{11})h_{11p}^{2}}{\kappa_{1}^{2}}
\leq \sum_{p\notin I}\frac{\sigma_{k}^{11,pp}h_{11p}^{2}}{\kappa_{1}}.
\end{equation}
Combining (\ref{Lemma 1 eqn 2}) and (\ref{Lemma 1 eqn 3}), we obtain
\begin{equation}\label{Lemma 1 eqn 4}
\sum_{p\notin I}\frac{\sigma_{k}^{pp}h_{11p}^{2}}{\kappa_{1}^{2}}
\leq 2\sum_{p\notin I}\frac{\sigma_{k}^{11}h_{11p}^{2}}{\kappa_{1}(\kappa_{1}-\ti{\kappa}_{p})}
+\sum_{p\notin I}\frac{\sigma_{k}^{11,pp}h_{11p}^{2}}{\kappa_{1}}.
\end{equation}
Then Lemma \ref{Lemma 1} follows from (\ref{Lemma 1 eqn 1}) and (\ref{Lemma 1 eqn 4}).
\end{proof}

The rest of the proof is very similar to \cite[Theorem 1.1]{GRW15}. For the reader's convenience, we give all the details here.

\begin{lemma}\label{Lemma 2}
For $\ve,\delta\in\left(0,\frac{1}{2}\right)$ and $1\leq l\leq k-1$, there exists a uniform constant $\delta'$ depending on $\ve$ and $\delta$ such that if $\kappa_{l}\geq\delta\kappa_{1}$ and $\kappa_{l+1}\leq\delta'\kappa_{1}$, then
\[
(1-2\ve)\frac{\sigma_{k}^{11}h_{111}^{2}}{\kappa_{1}^{2}} \leq -\frac{\sigma_{k}^{pp,qq}h_{pp1}h_{qq1}}{\kappa_{1}}
+2\sum_{p>1}\frac{\sigma_{k}^{pp}h_{pp1}^{2}}{\kappa_{1}(\kappa_{1}-\ti{\kappa}_{p})}+C\kappa_{1},
\]
for some uniform constant $C$.
\end{lemma}

\begin{proof}
Using \cite[(2.4)]{GRW15} (see also \cite[Lemma 3.2]{GLL12}), we have
\[
\begin{split}
& -\frac{\sigma_{k}^{pp,qq}h_{pp1}h_{qq1}}{\kappa_{1}}+\frac{\left(\sum_{p}\sigma_{k}^{pp}h_{pp1}\right)^{2}}{\kappa_{1}\sigma_{k}} \\
\geq {} & \frac{\sigma_{k}}{\kappa_{1}\sigma_{l}^{2}}\left[\left(\sum_{p}\sigma_{l}^{pp}h_{pp1}\right)^{2}-\sigma_{l}\sigma_{l}^{pp,qq}h_{pp1}h_{qq1}\right] \\
\geq {} & \frac{\sigma_{k}}{\kappa_{1}\sigma_{l}^{2}}\left[\sum_{p}\left(\sigma_{l}^{pp}h_{pp1}\right)^{2}
+\sum_{p\neq q}(\sigma_{l}^{pp}\sigma_{l}^{qq}-\sigma_{l}\sigma_{l}^{pp,qq})h_{pp1}h_{qq1}\right].
\end{split}
\]
Differentiating (\ref{Curvature equation}), we have
\[
\sum_{p}\sigma_{k}^{pp}h_{pp1} = h_{11}(d_{\nu}f)(e_{1})+(d_{X}f)(e_{1}),
\]
which implies
\[
\frac{\left(\sum_{p}\sigma_{k}^{pp}h_{pp1}\right)^{2}}{\kappa_{1}\sigma_{k}}
= \frac{\Big(h_{11}(d_{\nu}f)(e_{1})+(d_{X}f)(e_{1})\Big)^{2}}{\kappa_{1}f} \leq C\kappa_{1},
\]
assuming without loss of generality that $\kappa_{1}\geq 1$. Thus,
\begin{equation}\label{Lemma 2 eqn 1}
\begin{split}
& -\frac{\sigma_{k}^{pp,qq}h_{pp1}h_{qq1}}{\kappa_{1}}+C\kappa_{1} \\
\geq {} & \frac{\sigma_{k}}{\kappa_{1}\sigma_{l}^{2}}\left[\sum_{p}\left(\sigma_{l}^{pp}h_{pp1}\right)^{2}
+\sum_{p\neq q}(\sigma_{l}^{pp}\sigma_{l}^{qq}-\sigma_{l}\sigma_{l}^{pp,qq})h_{pp1}h_{qq1}\right].
\end{split}
\end{equation}

We claim
\begin{equation}\label{Lemma 2 eqn 2}
\begin{split}
& \sum_{p\neq q}(\sigma_{l}^{pp}\sigma_{l}^{qq}-\sigma_{l}\sigma_{l}^{pp,qq})h_{pp1}h_{qq1} \\
\geq {} & -\ve\sum_{p\leq l}(\sigma_{l}^{pp}h_{pp1})^{2}-\frac{C}{\ve}\sum_{p>l}(\sigma_{l}^{pp}h_{pp1})^{2}.
\end{split}
\end{equation}
When $l=1$, we have $\sigma_{1}^{pp,qq}=0$. Then the claim (\ref{Lemma 2 eqn 2}) follows from the Cauchy-Schwarz inequality. When $l>1$, we split the left-handed side of (\ref{Lemma 2 eqn 2}) into three terms:
\begin{equation}\label{Lemma 2 eqn 3}
\begin{split}
& \sum_{p\neq q}(\sigma_{l}^{pp}\sigma_{l}^{qq}-\sigma_{l}\sigma_{l}^{pp,qq})h_{pp1}h_{qq1} \\
= {} & \sum_{p\neq q;\,p,q\leq l}(\sigma_{l}^{pp}\sigma_{l}^{qq}-\sigma_{l}\sigma_{l}^{pp,qq})h_{pp1}h_{qq1} \\
& +2\sum_{p\leq l;\,q>l}(\sigma_{l}^{pp}\sigma_{l}^{qq}-\sigma_{l}\sigma_{l}^{pp,qq})h_{pp1}h_{qq1} \\
& +\sum_{p\neq q;\,p,q>l}(\sigma_{l}^{pp}\sigma_{l}^{qq}-\sigma_{l}\sigma_{l}^{pp,qq})h_{pp1}h_{qq1} \\
=: {} & T_{1}+T_{2}+T_{3}.
\end{split}
\end{equation}
By direct calculation and Newton's inequality, for $p\neq q$, we obtain (see \cite[(4.22)]{GRW15})
\begin{equation}\label{Lemma 2 eqn 4}
\sigma_{l}^{pp}\sigma_{l}^{qq}-\sigma_{l}\sigma_{l}^{pp,qq}
= \sigma_{l-1}^{2}(\kappa|pq)-\sigma_{l}(\kappa|pq)\sigma_{l-2}(\kappa|pq) \geq 0.
\end{equation}
Thus, for the term $T_{1}$ in (\ref{Lemma 2 eqn 3}), we have
\[
T_{1} \geq -\sum_{p\neq q;\,p,q\leq l}\sigma_{l-1}^{2}(\kappa|pq)|h_{pp1}h_{qq1}|.
\]
For $p\neq q$ and $p,q\leq l$, since $\kappa_{p},\kappa_{q}\geq\kappa_{l}\geq\delta\kappa_{1}$, $\kappa_{l+1}\leq\delta'\kappa_{1}$ and $\kappa_{i}>0$ for all $i$, then
\[
\sigma_{l-1}(\kappa|pq) \leq \frac{C\kappa_{1}\cdots\kappa_{l+1}}{\kappa_{p}\kappa_{q}}
\leq \frac{C\kappa_{l+1}}{\kappa_{q}}\cdot\frac{\kappa_{1}\cdots\kappa_{l}}{\kappa_{p}} \leq \frac{C\delta'\sigma_{l}^{pp}}{\delta}.
\]
Similarly, we have
\[
\sigma_{l-1}(\kappa|pq) \leq \frac{C\delta'\sigma_{l}^{qq}}{\delta}.
\]
Choosing $\delta'$ sufficiently small,
\[
\sigma_{l-1}^{2}(\kappa|pq) \leq \left(\frac{C\delta'}{\delta}\right)^{2}\sigma_{l}^{pp}\sigma_{l}^{qq}
\leq \frac{\ve\sigma_{l}^{pp}\sigma_{l}^{qq}}{2}.
\]
It then follows that
\begin{equation}\label{Lemma 2 eqn 5}
T_{1} \geq -\frac{\ve}{2}\sum_{p\neq q;\,p,q\leq l}|\sigma_{l}^{pp}h_{pp1}|\cdot|\sigma_{l}^{qq}h_{qq1}|
\geq -\frac{\ve}{2}\sum_{p\leq l}(\sigma_{l}^{pp}h_{pp1})^{2}.
\end{equation}
For the terms $T_{2}$ and $T_{3}$ in (\ref{Lemma 2 eqn 3}), using (\ref{Lemma 2 eqn 4}) and the Cauchy-Schwarz inequality,
\begin{equation}\label{Lemma 2 eqn 6}
\begin{split}
T_{2}+T_{3} \geq {} & -2\sum_{p\leq l;\,q>l}\sigma_{l}^{pp}\sigma_{l}^{qq}|h_{pp1}h_{qq1}|
-\sum_{p\neq q;\,p,q>l}\sigma_{l}^{pp}\sigma_{l}^{qq}|h_{pp1}h_{qq1}| \\
\geq {} & -\frac{\ve}{2}\sum_{p\leq l}(\sigma_{l}^{pp}h_{pp1})^{2}-\frac{C}{\ve}\sum_{p>l}(\sigma_{l}^{pp}h_{pp1})^{2}.
\end{split}
\end{equation}
Substituting (\ref{Lemma 2 eqn 5}) and (\ref{Lemma 2 eqn 6}) into (\ref{Lemma 2 eqn 3}), we obtain the claim (\ref{Lemma 2 eqn 2}).

Combining (\ref{Lemma 2 eqn 1}) and (\ref{Lemma 2 eqn 2}),
\begin{equation}\label{Lemma 2 eqn 7}
\begin{split}
& (1-\ve)\frac{\sigma_{k}(\sigma_{l}^{11})^{2}h_{111}^{2}}{\kappa_{1}\sigma_{l}^{2}} \leq
(1-\ve)\frac{\sigma_{k}}{\kappa_{1}\sigma_{l}^{2}}\sum_{p\leq l}\left(\sigma_{l}^{pp}h_{pp1}\right)^{2} \\
\leq {} & -\frac{\sigma_{k}^{pp,qq}h_{pp1}h_{qq1}}{\kappa_{1}}+\frac{C\sigma_{k}}{\ve\kappa_{1}\sigma_{l}^{2}}\sum_{p>l}(\sigma_{l}^{pp}h_{pp1})^{2}+C\kappa_{1}.
\end{split}
\end{equation}
Since $\kappa_{i}>0$ for all $i$ and $\kappa_{l+1}\leq\delta'\kappa_{1}$, we have
\[
\frac{\sigma_{k}}{\kappa_{1}\sigma_{k}^{11}} = \frac{\kappa_{1}\sigma_{k}^{11}+\sigma_{k}(\kappa|1)}{\kappa_{1}\sigma_{k}^{11}} \geq 1
\]
and
\[
\frac{\kappa_{1}\sigma_{l}^{11}}{\sigma_{l}}
= 1-\frac{\sigma_{l}(\kappa|1)}{\sigma_{l}}
\geq 1-\frac{C\kappa_{2}\cdots\kappa_{l+1}}{\kappa_{1}\cdots\kappa_{l}}
= 1-\frac{C\kappa_{l+1}}{\kappa_{1}} \geq 1-C\delta'.
\]
Thus, at the expense of decreasing $\delta'$, we obtain
\begin{equation}\label{Lemma 2 eqn 8}
\begin{split}
& (1-\ve)\frac{\sigma_{k}(\sigma_{l}^{11})^{2}h_{111}^{2}}{\kappa_{1}\sigma_{l}^{2}}
= (1-\ve)\frac{\sigma_{k}^{11}}{\kappa_{1}^{2}}\cdot\frac{\sigma_{k}}{\kappa_{1}\sigma_{k}^{11}}
\cdot\left(\frac{\kappa_{1}\sigma_{l}^{11}}{\sigma_{l}}\right)^{2}h_{111}^{2} \\
\geq {} & (1-\ve)(1-C\delta')^{2}\frac{\sigma_{k}^{11}h_{111}^{2}}{\kappa_{1}^{2}}
\geq (1-2\ve)\frac{\sigma_{k}^{11}h_{111}^{2}}{\kappa_{1}^{2}}.
\end{split}
\end{equation}
On the other hand, using $\kappa_{l}\geq\delta\kappa_{1}$ and $\kappa_{i}>0$ for all $i$, for $p>l$, we have
\[
\frac{\sigma_{l}^{pp}}{\sigma_{l}} \leq \frac{C\kappa_{1}\cdots\kappa_{l-1}}{\kappa_{1}\cdots\kappa_{l}}
\leq \frac{C}{\kappa_{l}} \leq \frac{C}{\delta\kappa_{1}}.
\]
This implies
\begin{equation}\label{Lemma 2 eqn 9}
\frac{C\sigma_{k}}{\ve\kappa_{1}\sigma_{l}^{2}}\sum_{p>l}(\sigma_{l}^{pp}h_{pp1})^{2}
= \frac{C}{\ve}\sum_{p>l}\left(\frac{\sigma_{l}^{pp}}{\sigma_{l}}\right)^{2}\cdot\frac{\sigma_{k}h_{pp1}^{2}}{\kappa_{1}}
\leq \frac{C}{\ve\delta^{2}}\sum_{p>l}\frac{\sigma_{k}h_{pp1}^{2}}{\kappa_{1}^{3}}.
\end{equation}
Since $\kappa_{l+1}\leq\delta'\kappa_{1}$ and $\kappa_{i}>0$ for all $i$, then for $l<p\leq k$,
\[
\frac{\sigma_{k}}{\kappa_{1}} \leq \frac{\delta'\sigma_{k}}{\kappa_{p}} \leq \frac{C\delta'\kappa_{1}\cdots\kappa_{k}}{\kappa_{p}} \leq C\delta'\sigma_{k}^{pp}.
\]
For $p>k$,
\[
\frac{\sigma_{k}}{\kappa_{1}} \leq \frac{\delta'\sigma_{k}}{\kappa_{k}} \leq C\kappa_{1}\cdots\kappa_{k-1} \leq C\delta'\sigma_{k}^{pp}.
\]
So $\frac{\sigma_{k}}{\kappa_{1}} \leq C\delta'\sigma_{k}^{pp}$ for $p>l$. It then follows that
\begin{equation}\label{Lemma 2 eqn 10}
\frac{C}{\ve\delta^{2}}\sum_{p>l}\frac{\sigma_{k}h_{pp1}^{2}}{\kappa_{1}^{3}}
\leq \frac{C\delta'}{\ve\delta^{2}}\sum_{p>l}\frac{\sigma_{k}^{pp}h_{pp1}^{2}}{\kappa_{1}^{2}}.
\end{equation}
Combining (\ref{Lemma 2 eqn 9}) and (\ref{Lemma 2 eqn 10}), and using $\kappa_{1}-\ti{\kappa}_{i}\leq\kappa_{1}+1$ for $i>1$, at the expense of decreasing $\delta'$, we see that
\begin{equation}\label{Lemma 2 eqn 11}
\frac{C\sigma_{k}}{\ve\kappa_{1}\sigma_{l}^{2}}\sum_{p>l}(\sigma_{l}^{pp}h_{pp1})^{2}
\leq \sum_{p>l}\frac{\sigma_{k}^{pp}h_{pp1}^{2}}{\kappa_{1}^{2}}
\leq 2\sum_{p>1}\frac{\sigma_{k}^{pp}h_{pp1}^{2}}{\kappa_{1}(\kappa_{1}-\ti{\kappa}_{p})},
\end{equation}
assuming without loss of generality that $\kappa_{1}\geq 1$. Substituting (\ref{Lemma 2 eqn 8}) and (\ref{Lemma 2 eqn 11}) into (\ref{Lemma 2 eqn 7}), we obtain
\[
(1-2\ve)\frac{\sigma_{k}^{11}h_{111}^{2}}{\kappa_{1}^{2}} \leq -\frac{\sigma_{k}^{pp,qq}h_{pp1}h_{qq1}}{\kappa_{1}}
+2\sum_{p>1}\frac{\sigma_{k}^{pp}h_{pp1}^{2}}{\kappa_{1}(\kappa_{1}-\ti{\kappa}_{p})}+C\kappa_{1},
\]
as required.
\end{proof}

\begin{lemma}\label{Lemma 3}
For $\delta\in\left(0,\frac{1}{2}\right)$ and $1\leq l\leq k-1$, there exists a uniform constant $\delta'$ and $C$ depending on $\delta$ such that if $\kappa_{l}\geq\delta\kappa_{1}$ and $\kappa_{l+1}\leq\delta'\kappa_{1}$, then $\kappa_{1}\leq C$.
\end{lemma}

\begin{proof}
Combining Lemma \ref{Calculation}, \ref{Lemma 1} and \ref{Lemma 2}, we obtain
\[
0 \geq -2\ve\frac{\sigma_{k}^{11}h_{111}^{2}}{\kappa_{1}^{2}}+\left(\frac{A}{C}-C\right)\sigma_{k}^{ii}h_{ii}^{2}-C\kappa_{1}-CA.
\]
Using (\ref{Calculation eqn 11}), we have
\[
-2\ve\frac{\sigma_{k}^{11}h_{111}^{2}}{\kappa_{1}^{2}}
= -2\ve A^{2}\sigma_{k}^{11}h_{11}^{2}\langle e_{1},X\rangle^{2}
\geq -C\ve A^{2}\sigma_{k}^{11}h_{11}^{2}.
\]
It then follows that
\[
0 \geq \left(\frac{A}{C}-C-C\ve A^{2}\right)\sigma_{k}^{ii}h_{ii}^{2}-C\kappa_{1}-CA.
\]
Using (\ref{Calculation eqn 9}), we have
\[
0 \geq \left(\frac{A}{C_{0}}-C_{0}-C_{0}\ve A^{2}\right)\kappa_{1}-C_{0}A,
\]
for some uniform constant $C_{0}$. Choosing $A=2C_{0}^{2}+C_{0}$ and $\ve=\frac{1}{A^{2}}$, we obtain $\kappa_{1}\leq C$, as required.
\end{proof}

We now complete the proof of Theorem \ref{Curvature estimate}. Set $\delta_{1}=\frac{1}{3}$, By Lemma \ref{Lemma 3}, there exists $\delta_{2}$ such that if $\kappa_{2}\leq\delta_{2}\kappa_{1}$, then $\kappa_{1}\leq C$. If $\kappa_{2}>\delta_{2}\kappa_{1}$, using Lemma \ref{Lemma 3} again, there exists $\delta_{3}$ such that if $\kappa_{3}\leq\delta_{3}\kappa_{1}$, then $\kappa_{1}\leq C$. Repeating the above argument, we obtain $\kappa_{1}\leq C$ or $\kappa_{k}>\delta_{k}\kappa_{1}$. In the latter case, since $\kappa_{1}\geq\kappa_{2}\geq\cdots\geq\kappa_{k}>\delta_{k}\kappa_{1}$ and $\kappa_{i}>0$ for all $i$, then
\[
\delta_{k}^{k}\kappa_{1}^{k} < \kappa_{1}\cdots\kappa_{k} \leq \sigma_{k} = f \leq C,
\]
which implies  $\kappa_{1}\leq C$, as required.
\end{proof}

\end{document}